\documentclass[preprint,review,10pt]{amsart}
\usepackage{amsmath,amssymb,amsfonts,xspace}
\newtheorem{theorem}{Theorem}[section]

\theoremstyle{definition}
\newtheorem{definition}[theorem]{Definition}
\newtheorem{example}[theorem]{Example}

\newtheorem{corollary}[theorem]{Corollary}

\theoremstyle{remark}

\numberwithin{equation}{section}

\begin{document}

\title{$\phi$-$(k,n)$-absorbing    and  $\phi$-$(k,n)$-absorbing primary   hyperideals in a Krasner $(m,n)$-hyperring
  }

\author{M. Anbarloei}
\address{Department of Mathematics, Faculty of Sciences,
Imam Khomeini International University, Qazvin, Iran.
}

\email{m.anbarloei@sci.ikiu.ac.ir }





\begin{abstract}
Various expansions of prime hyperideals have been studied in a Krasner $(m,n)$-hyperring $R$. For instance, a proper hyperideal $Q$ of $R$ is called weakly $(k,n)$-absorbing {\bf(}primary{\bf)} provided that for $r_1^{kn-k+1} \in R$, $g(r_1^{kn-k+1}) \in Q-\{0\}$ implies that there are $(k-1)n-k+2$ of the $r_i^,$s whose $g$-product is in $Q$ {\bf (}$g(r_1^{(k-1)n-k+2}) \in Q$ or a $g$-product of $(k-1)n-k+2$ of $r_i^,$s ,except $g(r_1^{(k-1)n-k+2})$, is in ${\bf r^{(m,n)}}(Q)${\bf )}.  
In this paper, we aim to extend the notions to the concepts of $\phi$-$(k,n)$-absorbing and $\phi$-$(k,n)$-absorbing primary hyperideals. 
Assume that $\phi$ is a function from $ \mathcal{HI}(R)$ to $\mathcal{HI}(R) \cup \{\varnothing\}$ such that $\mathcal{HI}(R)$ is the set of hyperideals of $R$ and $k$ is a positive integer. We call a proper hyperideal $Q$ of $R$ a $\phi$-$(k,n)$-absorbing {\bf (}primary {\bf )} hyperideal if for $r_1^{kn-k+1} \in R$, $g(r_1^{kn-k+1}) \in Q-\phi(Q)$ implies that there are $(k-1)n-k+2$ of the $r_i^,$s whose $g$-product is in $Q$ {\bf (}$g(r_1^{(k-1)n-k+2}) \in Q$ or a $g$-product of $(k-1)n-k+2$ of $r_i^,$s ,except $g(r_1^{(k-1)n-k+2})$, is in ${\bf r^{(m,n)}}(Q)$ {\bf )}. Several properties and characterizations of them are presented. 
\end{abstract}
\maketitle
\section{Introduction}
Extensions of prime and primary ideals to the context of $\phi$-prime and $\phi$-primary ideals are studied  in \cite{ander, darani}. Afterwards, Khaksari in \cite{khak} and Badawi et al. in \cite{bada} introduced $\phi$-$2$-prime and $\phi$-$2$-primary ideals, respectively. Let $R$ be a commutative ring. Suppose that $\phi$ is a function from $ \mathcal{I}(R)$ to $\mathcal{I}(R) \cup \{\varnothing\}$ where $\mathcal{I}(R)$ is the set of ideals of  $R$. A  proper ideal $I$ 
of $R$ is said to be a $\phi$-2-absorbing ideal if whenever  $x,y,z \in R$, with $xyz \in I-\phi(I)$ implies that $xy \in I$ or $xz \in I$ or $yz \in I$. Also, A  proper ideal $I$ 
of $R$ is called a $\phi$-2-absorbing primary ideal if for every $x,y,z \in R$, $xyz \in I-\phi(I)$ implies that $xy \in I$ or $xz \in {\bf r}(I)$ or $yz \in {\bf r}(I)$.

Hyperstructures are algebraic structures equipped with at least one multi-valued operation, called a hyperoperation. A hyperoperation on a nonempty set is a mapping from to the nonempty power set. Hundreds of papers and several books have been written on this topic (for more details see \cite{amer2, s2, s3,  davvaz1, s6, s5,s1, davvaz2, s4, s10, jian}).
An $n$-ary extension of algebraic structures is the most natural method for  deeper understanding of their fundamental properties. 
Mirvakili and Davvaz in \cite{cons} introduced $(m,n)$-hyperrings and gave several results in this respect.  They defined and described a generalization of the notion of a hypergroup  and a generalization of an $n$-ary group,   which is called $n$-ary hypergroup \cite{s9}.
 Some review of the $n$-ary structures  can be found in  in  \cite{l1, l2, l3, ma, rev1}. One important class of hyperrings, where the addition is a hyperoperation, while the multiplication is an ordinary binary operation,  is  Krasner hyperring.
 An extension  of the Krasner hyperrings, which is a subclass of $(m,n)$-hyperrings, was presented by Mirvakili and Davvaz \cite{d1}, which  is called Krasner $(m,n)$-hyperring. Some important
hyperideals namely Jacobson radical, nilradical, $n$-ary prime and primary hyperideals  and $n$-ary multiplicative subsets of Krasner $(m, n)$-hyperrings were defined  by Ameri and Norouzi in \cite{sorc1}. Afterward, the concept of $(k,n)$-absorbing (primary) hyperideals  was studied by Hila et al. \cite{rev2}. 
Norouzi et al.  gave a new definition for normal hyperideals in Krasner $(m,n)$-hyperrings, with respect to that one given in \cite{d1} and they showed that these hyperideals correspond to strongly regular relations \cite{nour}.  Direct limit of a direct system was  defined and analysed  by Asadi and Ameri in the category of Krasner $(m,n)$-hyperrigs \cite{asadi}.  The notion of $\delta$-primary  hyperideals in Krasner $(m,n)$-hyperrings, which unifies the prime and primary hyperideals under one frame, was presented in \cite{mah3}. Recently, Davvaz et al.  introduced new expansion classes, namely weakly $(k,n)$-absorbing  ( primary ) hyperideals in a Krasner $(m,n)$-hyperring \cite{weak}.

In this paper, we introduce and study the notions of $\phi$-$(k,n)$-absorbing and $\phi$-$(k,n)$-absorbing primary hyperideals in a commutative Krasner $(m,n)$-hyperring. A number of main results  are given to explain the general framework of these structures.
Among many results in this paper, it is shown (Theorem \ref{inclu}) that if $Q$ is a $\phi$-$(k,n)$-absorbing hyperideal of $R$, then $Q$ is a $\phi$-$(s,n)$-absorbing hyperideal for all $s \geq k$. Although every $\phi$-$(k,n)$-absorbing of a Krasner $(m,n)$-hyperring is  $\phi$-$(k,n)$-absorbing primary, Example \ref{exa2} shows that the converse  may not be always true. It is shown (Theorem \ref{kweak}) that 
$Q$  is   a $\phi$-$(k,n)$-absorbing primary hyperideal of $R$ if and only if  $Q/\phi(Q)$ is a weakly $(k,n)$-absorbing primary hyperideal of $R/\phi(Q)$. In Theorem \ref{weak5}, we show that if $Q$ is a $\phi$-$(k,n)$-absorbing primary hyperideal of $R$ but is not a $(k,n)$-absorbing primary, then $g(Q^{k(n-1)+1}) \subseteq \phi(Q)$. As a result of the theorem we conclude that if $Q$ is a $\phi$-$(k,n)$-absorbing primary hyperideal of $R$  that is not a $(k,n)$-absorbing primary hyperideal of $R$, then ${\bf r^{(m,n)}}(Q)={\bf r^{(m,n)}}(\phi(Q))$.
\section{Krasner $(m,n)$-hyperrings}
In this section, we summarize the preliminary definitions which are related to Krasner $(m,n)$-hyperrings. \\
Let $A$ be a non-empty set and $P^*(A)$  the
set of all the non-empty subsets of $A$.  An $n$-ary hyperoperation on $A$ is a map$f : A^n \longrightarrow P^*(A)$
 and the couple  $(A, f)$ is called an $n$-ary hypergroupoid. The notation $a_i^j$ will denote  the sequence $a_i, a_{i+1},..., a_j$ for $j \geq i$ and it is the empty symbol for $j < i$.
 If  $G_1,..., G_n$ are non-empty subsets of $A$, then we define
$f(G^n_1) = f(G_1,..., G_n) = \bigcup \{f(a^n_1) \ \vert \  a_i \in  G_i, 1 \leq i \leq  n \}.$ If $b_{i+1} =... = b_j = b$, we write $f(a^i_1, b^j_{i+1}, c^n_{j+1})=f(a^i_1, b^{(j-i)}, c^n_{j+1})$.
If $f$ is an $n$-ary hyperoperation, then $t$-ary hyperoperation $f_{(l)}$ is given by
\[f_{(l)}(a_1^{l(n-1)+1}) = f\bigg(f(..., f(f(a^n _1), a_{n+1}^{2n -1}),...), a_{(l-1)(n-1)+1}^{l(n-1)+1}\bigg).\]
where $t = l(n- 1) + 1$.
\begin{definition}
\cite{d1} $(R, f, g)$, or simply $R$, is defined as  a Krasner $(m, n)$-hyperring if the following statements hold: \\
(1) $(R, f$) is a canonical $m$-ary hypergroup;\\
(2) $(R, g)$ is a $n$-ary semigroup;\\
(3) The $n$-ary operation $g$ is distributive with respect to the $m$-ary hyperoperation $f$ , i.e., for every $a^{i-1}_1 , a^n_{ i+1}, x^m_ 1 \in R$, and $1 \leq i \leq n$,
\[g\bigg(a^{i-1}_1, f(x^m _1 ), a^n _{i+1}\bigg) = f\bigg(g(a^{i-1}_1, x_1, a^n_{ i+1}),..., g(a^{i-1}_1, x_m, a^n_{ i+1})\bigg);\]
(4) $0$ is a zero element  of the $n$-ary operation $g$, i.e., for each $a^n_ 1 \in R$ , 
$g(a_1^{i-1},0,a_{i+1}^n) = 0$.
\end{definition}
Throughout this paper,  $R$ denotes  a commutative Krasner $(m,n)$-hyperring with the scalar identity $1$.

A non-empty subset $T$ of $R$ is called a subhyperring of $R$ if $(T, f, g)$ is a Krasner $(m, n)$-hyperring. 
The non-empty subset $I$ of $R$  is a hyperideal  if $(I, f)$ is an $m$-ary subhypergroup
of $(R, f)$ and $g(x^{i-1}_1, I, x_{i+1}^n) \subseteq I$, for each $x^n _1 \in  R$ and  $1 \leq i \leq n$.

\begin{definition}
 \cite{sorc1} Let  $I$ be a proper hyperideal of  $R$. $I$ refers to   a prime hyperideal if for hyperideals $I_1^n$ of $R$, $g(I_1^ n) \subseteq P$ implies $I_i \subseteq I$ for some $1 \leq i \leq n$.
\end{definition}
 
Lemma 4.5 in \cite{sorc1} shows that   the proper hyperideal $I$ of $R$ is  prime  if for all $a^n_ 1 \in R$, $g(a^n_ 1) \in I$ implies  $a_i \in I$ for some $1 \leq i \leq n$. 
\
\begin{definition} \cite{sorc1}  The radical of the proper the hyperideal $I$ of $R$, denoted by ${\bf r^{(m,n)}}(I)$
is the intersection of  all  prime hyperideals of $R$ containing $I$. If the set of all prime hyperideals  which contain $I$ is empty, then ${\bf r^{(m,n)}}(I)=R$.
\end{definition}
 It was shown (Theorem 4.23 in \cite{sorc1})  that if $a \in {\sqrt I}^{(m,n)}$ then 
 there exists $s \in \mathbb {N}$ with $g(a^ {(s)} , 1_R^{(n-s)} ) \in I$ for $s \leq n$, or $g_{(l)} (a^ {(s)} ) \in I$ for $s = l(n-1) + 1$.
\begin{definition}
\cite{sorc1} A proper hyperideal  $I$ of $R$ is     primary  if $g(a^n _1) \in I$ and $a_i \notin I$ implies $g(a_1^{i-1}, 1_R, a_{ i+1}^n) \in {\bf r^{(m,n)}}(I)$ for some $1 \leq i \leq n$.
\end{definition}
Theorem 4.28 in \cite{sorc1} shows that the radical of a primary hyperideal of $R$ is prime.
 \begin{definition} 
\cite{rev2} Let $I$ be a proper hyperideal of  $R$. $I$ refers to an 
\begin{itemize}
\item[\rm(1)]~ $(k,n)$-absorbing hyperideal  if for  $r_1^{kn-k+1} \in R$, $g(r_1^{kn-k+1}) \in I$ implies that there exist $(k-1)n-k+2$ of the $r_i^,$s whose $g$-product is in $I$. In this case, if $k=1$, then $I$ is an $n$-ary prime hyperideal of $R$. If $n=2$ and $k=1$, then $I$ is a classic prime hyperideal of $R$.
\item[\rm(2)]~ $(k,n)$-absorbing primary hyperideal  if for  $r_1^{kn-k+1} \in R$, $g(r_1^{kn-k+1}) \in I$ implies that $g(r_1^{(k-1)n-k+2}) \in I$ or a $g$-product of $(k-1)n-k+2$ of the $r_i^,$s, except $g(r_1^{(k-1)n-k+2})$, is in ${\bf r^{(m,n)}}(I)$.
\end{itemize}
\end{definition}

\section{$\phi$-$(k,n)$-absorbing hyperideals}
In his paper \cite{weak}, Davvaz et al. introduced a generalization of $n$-ary  prime hyperideals in a Krasner $(k,n)$-hyperring, which they defined as weakly $(k,n)$-absorbing hyperideals. In this section, we generalize this notion to the context of $\phi$-$(k,n)$-absorbing hyperideals.
\begin{definition}
Assume that $\mathcal{HI}(R)$ is the set of hyperideals of  $R$ and $\phi: \mathcal{HI}(R) \longrightarrow \mathcal{HI}(R) \cup \{\varnothing\}$ is a function. Let $k$ be a positive integer. A proper hyperideal $Q$ of $R$ is said to be $\phi$-$(k,n)$-absorbing provided that for $r_1^{kn-k+1} \in R$,  $g(r_1^{kn-k+1}) \in Q-\phi(Q)$ implies that there are $(k-1)n-k+2$ of the $r_i^,$s whose $g$-product is in $Q$.
\end{definition}
\begin{example}
Consider the Krasner $(2,2)$-hyperring $R=\{0,1,x\}$ with the hyperaddition and multiplication defined by
$\vspace{0.5cm}$

$\hspace{3cm}$
$\begin{tabular}{|c||c|c|c|} 
\hline $+$ & $0$ & $1$ & $x$
\\ \hline \hline $0$ & $0$ & $1$ & $x$
\\ \hline $1$ & $1$ & $R$ &  $1$ 
\\ \hline $x$ & $x$ & $1$ & $\{0,x\}$ 
\\ \hline
\end{tabular}$
$\hspace{0.5cm}$
$\begin{tabular}{|c||c|c|c|} 
\hline $\cdot$ & $0$ & $1$ & $x$
\\ \hline \hline $0$ & $0$ & $0$ & $0$
\\ \hline $1$ & $0$ & $1$ & $x$ 
\\ \hline $x$ & $0$ & $x$ & $0$ 
\\ \hline
\end{tabular}$
$\vspace{0.5cm}$

Assume that   $\phi$ is  a function from $\mathcal{HI}(R)$ to $\mathcal{HI}(R) \cup \{\varnothing\}$ defined $\phi(I)=g(I^{(2)})$ for  $I \in \mathcal{HI}(R)$. Then the hyperideal $Q=\{0,x\}$ is a $\phi$-$(2,2)$-absorbing  hyperideal of $R$.
\end{example}
\begin{theorem} \label{zir}
Let $\phi_1, \phi_2: \mathcal{HI}(R) \longrightarrow \mathcal{HI}(R) \cup \{\varnothing\}$ be  two  functions such that for all $I \in  \mathcal{HI}(R)$, $\phi_1(I) \subseteq \phi_2(I)$. If $Q$ is a $\phi_1$-$(k,n)$-absorbing hyperideal of $R$, then $Q$ is a $\phi_2$-$(k,n)$-absorbing hyperideal.
\end{theorem}
\begin{proof}
Suppose that $g(r_1^{kn-k+1}) \in Q-\phi_2(Q)$ for $r_1^{kn-k+1} \in R$. From $\phi_1(Q) \subseteq \phi_2(Q)$, it follows that $g(r_1^{kn-k+1}) \in Q-\phi_1(Q)$. Since $Q$ is a $\phi_1$-$(k,n)$-absorbing hyperideal of $R$, we conclude that there are $(k-1)n-k+2$ of the $r_i^,$s whose $g$-product is in $Q$, as needed.
\end{proof}
\begin{theorem} \label{inclu}
Let $\phi: \mathcal{HI}(R) \longrightarrow \mathcal{HI}(R) \cup \{\varnothing\}$ be a function. If $Q$ is a $\phi$-$(k,n)$-absorbing hyperideal of $R$, then $Q$ is a $\phi$-$(s,n)$-absorbing hyperideal for all $s \geq k$.
\end{theorem}
\begin{proof}
Let us use the induction on $k$ that if  $Q$ is $\phi$-$(k,n)$-absorbing hyperideal of $R$, then $Q$ is $\phi$-$(k+1,n)$-absorbing. Assume that $Q$ is $\phi$-$(2,n)$-absorbing and $g(r_1^{2n-2},g(r_{2n-1}^{3n-2})) \in Q-\phi(Q)$ for some $r_1^{3n-2} \in R$. Since $Q$ is $\phi$-$(2,n)$-absorbing, then there are $n$  of the $r_i^,$s except $g(r_{2n-1}^{3n-2})$ whose $g$-product is in $Q$ and so there are $2n-1$ of the $r_i^,$s whose $g$-product is in $Q$. This shows that $Q$ is $\phi$-$(3,n)$-absorbing. Assume that $Q$ is $\phi$-$(k,n)$-absorbing and $g(g(r_1^{2n-2}),r_{2n-1}^{(k+1)n-(k+1)+1}) \in Q-\phi(Q)$ for some $r_1^{(k+1)n-(k+1)+1} \in R$. Since $Q$ is $\phi$-$(k,n)$-absorbing, we conclude that $g(g(r_1^{2(n-1)}),r_{2n-1},\cdots, \widehat{r_i}, \cdots, r_{(k+1)n-(k+1)+1}) \in Q$ for some $2(n-1) \leq i \leq (k+1)n-(k+1)+1$ or $g(r_{2n-1}^{(k+1)n-(k+1)+1}) \in Q$. The former case shows that $Q$ is $\phi$-$(k+1,n)$-absorbing. In the latter case, we obtain $g(r_1^{n-1},r_{2n-1}^{(k+1)n-(k+1)+1}) \in Q$ since  $g(r_1^{2(n-1)}) \in Q$. Thus $Q$ is $\phi$-$(k+1,n)$-absorbing.
\end{proof}
Recall from \cite{car} that if $(R_1, f_1, g_1)$ and $(R_2, f_2, g_2)$ are two Krasner $(m,n)$-hyperrings such that $1_{R_1}$ and $1_{R_2}$ are scalar identitis of $R_1$ and $R_2$, respectively, then 
  $(R_1 \times R_2, f_1\times f_2 ,g_1 \times g_2 )$ is a Krasner $(m, n)$-hyperring where

$f=f_1 \times f_2((a_{1}, b_{1}),\cdots,(a_m,b_m)) = \{(a,b) \ \vert \ \ a \in f_1(a_1^m), b \in f_2(b_1^m) \}$

$g=g_1 \times g_2 ((x_1,y_1),\cdots,(x_n,y_n)) =(g_1(x_1^n),g_2(y_1^n)) $,\\
for all $a_1^m,x_1^n \in R_1$ and $b_1^m,y_1^n \in R_2$
 \begin{theorem} \label{car}
 Let $(R_i,f_i,g_i)$ be a commutative
Krasner $(m,n)$-hyperring  for each $1 \leq i \leq kn-k+1$ and $\phi_i: \mathcal{HI}(R_i) \longrightarrow \mathcal{HI}(R_i) \cup \{\varnothing\}$ be a function. Let $Q_i$ be a  hyperideal of $R_i$ for each $1 \leq i \leq kn-k+1$ and $\phi=\phi_1 \times \cdots \times \phi_{kn-k+1}$. If $Q=Q_1 \times \cdots \times Q_{kn-k+1}$ is a $\phi$-$(k+1,n)$-absorbing hyperideal of $R=R_1 \times \cdots \times R_{kn-k+1}$, then $Q_i$ is a $\phi_i$-$(k,n)$-absorbing hyperideal of $R_i$ and $Q_i \neq R_i$ for all $1 \leq i \leq kn-k+1$. 
 \end{theorem}
\begin{proof}
Let $r_1^{kn-k+1} \in R_i$ such that $g(r_1^{kn-k+1}) \in Q_i-\phi_i(Q_i)$. Suppose by contradiction that $Q_i$ is not a $\phi_i$-$(k,n)$-absorbing  hyperideal of $R_i$. Define 

$a_1=(1_{R_1},\cdots, 1_{R_{i-1}},r_1,1_{R_{i+1}},\cdots, 1_{R_{kn-k+1}}),$

$ a_2=(1_{R_1},\cdots, 1_{R_{i-1}},r_2,1_{R_{i+1}},\cdots, 1_{R_{kn-k+1}}),$

$\vdots$

$a_{kn-k+1}=(1_{R_1},\cdots, 1_{R_{i-1}},r_{kn-k+1},1_{R_{i+1}},\cdots,1_{R_{kn-k+1}}), $

$a_{kn-k}=(1_{R_1},\cdots, 1_{R_{i-1}},1_{R_i},1_{R_{i+1}},\cdots,1_{R_{kn-k+1}}),$

$ a_{(k+1)n-(k+1)+1}=(0,\cdots, 0,1_{R_i},0,\cdots,0)$.\\
Hence $g(a_1^{(k+1)n-(k+1)+1)}) \in Q-\phi(Q)$ but $g(a_1^{kn-k+1}) \notin Q$. Since $Q$ is a $\phi$-$(k+1,n)$-absorbing hyperideal of $R$,  we conclude that one of $g$-productions of $kn-k+1$ of $a_i^,$s except $g(a_1^{(k+1)n-(k+1)+1})$ is in $Q$. This implies that there exist $(k-1)n-k+2$ of $r_i^,$s whose $g$-product is in $Q_i$ which is a contradiction. Consequently,  $Q_i$ is a $\phi_i$-$(k,n)$-absorbing hyperideal of $R_i$.
\end{proof}
 Assume that $(R_1, f_1, g_1)$ and $(R_2, f_2, g_2)$ are two Krasner $(m, n)$-hyperrings. Recall from \cite{d1} that a mapping
$h : R_1 \longrightarrow R_2$ is called a homomorphism if for all $a^m _1 \in R_1$ and $b^n_ 1 \in R_1$ we have
$(1) h(f_1(a_1,..., a_m)) = f_2(h(a_1),...,h(a_m))$, 
$(2) h(g_1(b_1,..., b_n)) = g_2(h(b_1),...,h(b_n)). $
Moreover, recall from \cite{Jaber } that a function $\phi: \mathcal{HI}(R) \longrightarrow \mathcal{HI}(R) \cup \{\varnothing\}$ is called a reduction function of $\mathcal{HI}(R)$ if $\phi(P) \subseteq P$ and $P \subseteq Q$ implies that $\phi(P) \subseteq \phi(Q)$ for all $P,Q \in \mathcal{HI}(R)$. Now, assume that $R_1$ and $R_2$ are two Krasner $(m,n)$-hyperring such that  $h: R_1 \longrightarrow R_2$ is a homomorphism. Suppose that $\phi_1$ and $\phi_2$ are two reduction functions of $\mathcal{HI}(R_1)$ and $\mathcal{HI}(R_2)$, respectively.  If $\phi_1(h^{-1}(I_2))=h^{-1}(\phi_2(I_2))$ for all $I_2 \in \mathcal{HI}(R_2)$, then we say $h$ is a $\phi_1$-$\phi_2$-homomorphism. Let $h$ be a $\phi_1$-$\phi_2$-epimorphism from $R_1$ to $R_2$ and let $I_1$ be a hyperideal of $R_1$ with $Ker (h) \subseteq I_1$. It is easy to see that $\phi_2(h(I_1))=h(\phi_1(I_1))$.
\begin{example}
Let $R_1$ and $R_2$ be two Krasner $(m,n)$-hyperring and $\phi_1$ and $\phi_2$ be two empty reduction functions of $\mathcal{HI}(R_1)$ and $\mathcal{HI}(R_2)$, respectively.  Then every homorphism $h$ from $R_1$ to $R_2$ is a $\phi_1$-$\phi_2$-homomorphism.
\end{example}
\begin{theorem} \label{homo1}
Let $h:R_1 \longrightarrow R_2$ be a $\phi_1$-$\phi_2$-homomorphism, where $\phi_1$ and $\phi_2$ are two reduction functions of $\mathcal{HI}(R_1)$ and $\mathcal{HI}(R_2)$, respectively. Then:
\begin{itemize} 
\item[\rm(1)]~ If $Q_2$ is a $\phi_2$-$(k,n)$-absorbing  hyperideal of $R_2$, then $h^{-1}(Q_2)$ is a $\phi_1$-$(k,n)$-absorbing  of $R_1$. 
\item[\rm(2)]~If $h$ is surjective and $Q_1$ is a $\phi_1$-$(k,n)$-absorbing  hyperideal of $R_1$ with $Ker (h) \subseteq Q_1$, then $h(Q_1)$ is a $\phi_2$-$(k,n)$-absorbing  hyperideal of $R_2$.
\end{itemize} 
\end{theorem}
\begin{proof}
 $(1)$ Let $Q_2$ be a $\phi_2$-$(k,n)$-absorbing  hyperideal of $R_2$ and $g(r_1^{kn-k+1}) \in h^{-1}(Q_2)-\phi_1(h^{-1}(Q_2))$ for some $r_1^{kn-k+1} \in R_1$. Then we get $h(g(r_1^{kn-k+1}))=g(h(r_1),\cdots,h(r_{kn-k+1})) \in Q_2-\phi_2(Q_2)$. Since $Q_2$ is a $\phi_2$-$(k,n)$-absorbing  hyperideal of $R_2$, we conclude that the image of $h$ of ${(k-1)n-k+2}$ of $r_i^,$s whose $g$-product is in $Q_2$. Then there exist ${(k-1)n-k+2}$ of $r_i^,$s whose $g$-product is in $h^{-1}(Q_2)$. Thus $h^{-1}(Q_2)$ is a $\phi_1$-$(k,n)$-absorbing  of $R_1$.

 $(2)$ Suppose that  $Q_1$ is a $\phi_1$-$(k,n)$-absorbing  hyperideal of $R_1$ with $Ker (h) \subseteq Q_1$ and $h$ is surjective. Let  $g(s_1^{kn-k+1}) \in h(Q_1)-\phi_2(h(Q_1))$ for some $s_1^{kn-k+1} \in R_2$. Then there exist $r_i \in R_1$ for every $1 \leq i \leq kn-k+1$ such that  $h(r_i)=s_i$. Hence we get $h(g(r_1^{kn-k+1})=g(h(r_1),\cdots,h(r_{kn-k+1}))=g(s_1^{kn-k+1}) \in h(Q_1)$. Since $Ker (h) \subseteq Q_1$ and $h$ is a $\phi_1$-$\phi_2$-epimorphism, we have $g(r_1^{kn-k+1}) \in Q_1-\phi_1(Q_1)$. Since $Q_1$ is a $\phi_1$-$(k,n)$-absorbing  hyperideal of $R_1$, there are $(k-1)n-k+2$ of $r_i^,$s whose $g$-product is in $Q_1$. Now, since $h$ is a homomorphism, we are done.
\end{proof}
Let $P$ be a hyperideal of $R$. Then the set $R/P=\{f(a_1^{i-1},P,a_{i+1}^m) \ \vert \ a_1^{i-1},a_{i+1}^m \in R$\} with $m$-ary hyperoperation $f$ and $n$-operation $g$ is the quotient Krasner $(m,n)$-hyperring of $R$ by $P$. Theorem 3.2 in \cite{sorc1} shows that the projection map $\pi$ from $R$ to $R/P$, defined by $\pi(r)=f(r,P,0^{(m-2)})$, is homomorphism.
Let $P$ be  a hyperideal of $R$ and $\phi$ be a reduction function of $\mathcal{HI}(R)$. Then the function $\phi_q$ from $\mathcal{HI}(R/P)$ to $\mathcal{HI}(R/P) \cup \{\varnothing\}$ defined by $\phi_q(I/P)=\phi(I)/P$ is a reduction function. Now, we have the following theorem as a result of Theorem \ref{homo1} that is easily verified, and hence we omit the proof.
\begin{theorem}
Let $Q$ and $P$ be  two  hyperideals of $R$ and $\phi$ be a reduction function of $\mathcal{HI}(R)$ such that $P \subseteq \phi(Q) \subseteq Q$. If $Q$ is a $\phi$-$(k,n)$-absorbing  hyperideal of $R$, then $Q/P$ is a $\phi_q$-$(k,n)$-absorbing  hyperideal of $R/P$.
\end{theorem}
\section{$\phi$-$(k,n)$-absorbing primary hyperideals}
\begin{definition}
Suppose that $\mathcal{HI}(R)$ is the set of hyperideals of  $R$ and $\phi: \mathcal{HI}(R) \longrightarrow \mathcal{HI}(R) \cup \{\varnothing\}$ is a function. Let $k$ be a positive integer. A proper hyperideal $Q$ of $R$ is called  $\phi$-$(k,n)$-absorbing primary if   $g(r_1^{kn-k+1}) \in Q-\phi(Q)$ for $r_1^{kn-k+1} \in R$ implies that  $g(r_1^{(k-1)n-k+2}) \in Q$ or a $g$-product  of $(k-1)n-k+2$ of  $r_i^,$s ,except $g(r_1^{(k-1)n-k+2})$, is in ${\bf r^{(m,n)}}(Q)$.
\end{definition} 
\begin{example} \label{exa}
Every $\phi$-$(k,n)$-absorbing of a Krasner $(m,n)$-hyperring is  $\phi$-$(k,n)$-absorbing primary.
\end{example} 
The converse  may not be always true as it is shown in the following example. 
\begin{example} \label{exa2}
Consider the Krasner $(2,2)$-hyperring  $R=[0,1]$ with the  $2$-ary hyperoperation defined as 
$a \oplus b=\{max\{a, b\}\}$ if $a \neq b$ and $a \oplus b=[0,a]$
if $a =b$
and multiplication is the usual multiplication on real numbers. Suppose  that   $\phi$ is  a function from $\mathcal{HI}(R)$ to $\mathcal{HI}(R) \cup \{\varnothing\}$ defined $\phi(I)=\cap_{i=1}^{\infty} g(I^{(i)})$ for  $I \in \mathcal{HI}(R)$.Then the hyperideal $Q=[0,0.5]$ is a  $\phi$-$(2,2)$-primary hyperideal of $R$ but it is not $\phi$-$(2,2)$-absorbing.
\end{example}
The next theorem provides us how to determine $\phi$-$(k,n)$-absorbing primary hyperideal to be $(k,n)$-absorbing primary.
\begin{theorem}
Assume that $Q$ is a hyperideal of $R$ and $\phi: \mathcal{HI}(R) \longrightarrow \mathcal{HI}(R) \cup \{\varnothing\}$ is a reduction function such that $\phi(Q)$ is a $(k,n)$-absorbing primay huperideal of $R$. If $Q$ is a $\phi$-$(k,n)$-absorbing primary hyperideal of $R$, then $Q$ is a $(k,n)$-absorbing primary hyperideal of $R$.
\end{theorem}
\begin{proof}
Let $r_1^{kn-k+1} \in R$ such that $g(r_1^{kn-k+1}) \in Q$ such that $g(r_1^{(k-1)n-k+2}) \notin Q$. Assume that $g(r_1^{kn-k+1}) \in \phi(Q)$. Since $\phi(Q)$ is a $(k,n)$-absorbing primay huperideal of $R$ and $g(r_1^{(k-1)n-k+2}) \notin \phi(Q)$, we conclude that a $g$-product of $(k-1)n-k+2$ of the $r_i^,$s, except $g(r_1^{(k-1)n-k+2})$ is in ${\bf r^{(m,n)}}(\phi(Q)) \subseteq {\bf r^{(m,n)}}(Q)$, as needed. Suppose that  $g(r_1^{kn-k+1}) \notin \phi(Q)$. Since $Q$ is a $\phi$-$(k,n)$-absorbing primary hyperideal of $R$, we are done.
\end{proof}
In the following, the relationship between a $\phi$-$(k,n)$-absorbing primary hyperideal of $R$ and its radical is considered.
\begin{theorem}
Let $Q$ be a hyperideal of $R$ and $\phi: \mathcal{HI}(R) \longrightarrow \mathcal{HI}(R) \cup \{\varnothing\}$ be a  function such that ${\bf r^{(m,n)}}(\phi(Q))=\phi({\bf r^{(m,n)}}(Q))$. If $Q$ is a $\phi$-$(k,n)$-absorbing primary hyperideal of $R$, then ${\bf r^{(m,n)}}(Q)$ is a $\phi$-$(k,n)$-absorbing  hyperideal of $R$.
\end{theorem}
\begin{proof}
Let $r_1^{kn-k+1} \in R$ such that $g(r_1^{kn-k+1}) \in {\bf r^{(m,n)}}(Q)-\phi({\bf r^{(m,n)}}(Q))$. Assume that all products of $(k-1)n-k+2$ of the $r_i^,$s except $g(r_1^{(k-1)n-k+2})$ are not in ${\bf r^{(m,n)}}(Q)$. Since $g(r_1^{kn-k+1}) \in {\bf r^{(m,n)}}(Q)$, then there exist $s \in \mathbb{Z}^+$ with $g(g(r_1^{kn-k+1})^{(s)},1^{(n-s)}) \in Q$, for $s \leq n$ or $g_{(l)}(g(r_1^{kn-k+1})^{(s)}) \in Q$, for $s>n$, $s=l(n-1)+1$. In the former case, we get $g(g(r_1)^{(s)},g(r_2)^{(s)},\cdots,g(r_{kn-k+1})^{(s)},1^{(n-s)}) \in Q$. If $g(g(r_1)^{(s)},g(r_2)^{(s)},\cdots,g(r_{kn-k+1})^{(s)},1^{(n-s)}) \in \phi(Q)$, we obtain $g(r_1^{kn-k+1}) \in {\bf r^{(m,n)}}(\phi(Q))=\phi({\bf r^{(m,n)}}(Q))$, a contradiction.  Since $Q$ is a $\phi$-$(k,n)$-absorbing primary hyperideal of $R$, then $g(g(r_1)^{(s)},g(r_2)^{(s)},\cdots,g(r_{(k-1)n-k+2})^{(s)}),1^{(n-s)})=g(g(r_1^{(k-1)n-k+2)})^{(s)},1^{(n-s)}) \in Q$ which means $g(r_1^{(k-1)n-k+2}) \in {\bf r^{(m,n)}}(Q)$. For the other case, we have a similar argument. Consequently, ${\bf r^{(m,n)}}(Q)$ is a $\phi$-$(k,n)$-absorbing  hyperideal of $R$.
\end{proof}
\begin{theorem}
Assume that $\phi: \mathcal{HI}(R) \longrightarrow \mathcal{HI}(R) \cup \{\varnothing\}$ is a function. If $Q$ is a $\phi$-$(k,n)$-absorbing primary hyperideal of $R$, then $Q$ is a $\phi$-$(s,n)$-absorbing primary hyperideal for all $s \geq k$.
\end{theorem}
\begin{proof}
Let $Q$ be a $\phi$-$(k,n)$-absorbing primary hyperideal of $R$. Suppose that $g(g(r_1^{n+2}),r_{n+3}^{(k+1)n-(k+1)+1}) \in Q-\phi(Q)$ for some $r_1^{(k+1)n-(k+1)+1} \in R$. Put $g(r_1^{n+2})=a_1$. Then we conclude that $g(a_1,\cdots,r_{(k+1)n-(k+1)+1}) \in Q$ or a $g$-product of $kn-k+1$ of the $r_i^,$s, except $g(a_1,\cdots, r_{(k+1)n-(k+1)+1})$ is in ${\bf r^{(m,n)}}(Q)$ as $Q$ is a $\phi$-$(k,n)$-absorbing primary hyperideal of $R$. Since ${\bf r^{(m,n)}}(Q)$ is a hyperideal of $R$ and $r_1^{n+2} \in R$, we conclude that $g(r_1,r_{n+3},\cdots,r_{(k+1)n-(k+1)+1}) \in {\bf r^{(m,n)}}(Q)$ or $ \cdots$ or $g(r_{n+2},r_{n+3},\cdots,r_{(k+1)n-(k+1)+1}) \in {\bf r^{(m,n)}}(Q)$ and so $Q$ is $(k+1,n)$-absorbing primary.
\end{proof}
\begin{theorem}
Let $\phi_1, \phi_2: \mathcal{HI}(R) \longrightarrow \mathcal{HI}(R) \cup \{\varnothing\}$ be  two  functions such that for all $I \in  \mathcal{HI}(R)$, $\phi_1(I) \subseteq \phi_2(I)$. If $Q$ is a $\phi_1$-$(k,n)$-absorbing primary hyperideal of $R$, then $Q$ is a $\phi_2$-$(k,n)$-absorbing primary hyperideal.
\end{theorem}
\begin{proof}
It is proved in a similar way to Theorem \ref{zir}.
\end{proof}
\begin{theorem}
Let $\phi: \mathcal{HI}(R) \longrightarrow \mathcal{HI}(R) \cup \{\varnothing\}$ be a function. If $Q$ is a $\phi$-$(1,n)$-absorbing primary hyperideal of $R$, then $Q$ is a $\phi$-$(2,n)$-absorbing primary hyperideal.
\end{theorem}
\begin{proof}
Let $Q$ be a $\phi$-$(1,n)$-absorbing primary hyperideal and $g(g(r_1^n),\cdots,r_{2n-1}) \in Q-\phi(Q)$ for some $r_1^{2n-1} \in R$. Then we get $g(r_1^n) \in Q$ or $g(r_{n+1}^{2n-1}) \in {\bf r^{(m,n)}}(Q)$.  By definition of hyperideal, we conclude that  $g(r_1,r_{n+1},\cdots,r_{2n-1}) \in  {\bf r^{(m,n)}}(Q)$ or $\cdots$ or $g(r_1,r_{n+1},\cdots,r_{2n-1}) \in {\bf r^{(m,n)}}(Q)$ since $r_1^n \in R$. Consequently, $Q$ is a $\phi$-$(2,n)$-absorbing primary hyperideal of $R$.
\end{proof}
Let $Q$ be a proper hyperideal of $R$ and $\phi: \mathcal{HI}(R) \longrightarrow \mathcal{HI}(R) \cup \{\varnothing\}$ be a function. $Q$ refers to a strongly $\phi$-$(k,n)$-absorbing primary hyperideal of $R$ if $g(Q_1^{kn-k+1}) \subseteq Q-\phi(Q)$ for some hyperideals $Q_1^{kn-k+1}$ of $R$ implies that $g(Q_1^{(k-1)n-k+2}) \subseteq Q$ or a $g$-product of $(k-1)n-k+2$ of $Q_i^,$s, except $g(Q_1^{(k-1)n-k+2}),$ is a subset of ${\bf r^{(m,n)}}(Q)$. In the sequel, we assume that all $\phi$-$(k,n)$-absorbing primary hyperideals of $R$ are strongly $\phi$-$(k,n)$-absorbing primary hyperideals.
Recall from \cite{weak} that a proper hyperideal $Q$ of $R$ is called  weakly $(k,n)$-absorbing primary if   $0 \neq g(r_1^{kn-k+1}) \in Q$ for $r_1^{kn-k+1} \in R$ implies that  $g(r_1^{(k-1)n-k+2}) \in Q$ or a $g$-product  of $(k-1)n-k+2$ of  $r_i^,$s ,except $g(r_1^{(k-1)n-k+2})$, is in ${\bf r^{(m,n)}}(Q)$.
\begin{theorem} \label{weak}
Suppose that $Q$ is a proper hyperideal of a commutative Krasner $(m,2)$-hyperring $R$ and $\phi: \mathcal{HI}(R) \longrightarrow \mathcal{HI}(R) \cup \{\varnothing\}$ is a function. Then the followings are equivalent:

\begin{itemize} 
\item[\rm(1)]~ $Q$  is   a $\phi$-$(2,2)$-absorbing primary hyperideal of $R$.  
\item[\rm(2)]~ $Q/\phi(Q)$ is a weakly $(2,2)$-absorbing primary hyperideal of $R/\phi(Q)$.



\end{itemize}
\end{theorem}
\begin{proof}
$(1) \Longrightarrow (2) $ Let  $Q$  be  $\phi$-$(2,2)$-absorbing primary  and   for $a_{11}^{1m},a_{21}^{2m},a_{31}^{3m}\in R$, 

$\hspace{1cm}\phi(Q) \neq g(f(a_{11}^{1(i-1)},\phi(Q),a_{1(i+1)}^{1m}),f(a_{21}^{2(i-1)},\phi(Q),a_{2(i+1)}^{2m}),$ 

$\hspace{2.5cm}f(a_{31}^{3(i-1)},\phi(Q),a_{3(i+1)}^{3m}))$

$\hspace{1.8cm}=f(g(a_{11}^{31}),\cdots,g(a_{1(i-1)}^{3(i-1)}),\phi(Q),g(a_{1(i+1)}^{3(i+1)}),\cdots,g(a_{1m}^{3m}))$

$\hspace{1.8cm} \in Q/\phi(Q)$. \\
Then  

$\hspace{1cm} f(g(a_{11}^{31}),\cdots, g(a_{1(i-1)}^{3(i-1)}),0,g(a_{1(i+1)}^{3(i+1)}),\cdots,g(a_{1m}^{3m}))$  

$\hspace{1cm}=g(f(a_{11}^{1(i-1)},0,a_{1(i+1)}^{1m}),f(a_{21}^{2(i-1)},0,a_{2(i+1)}^{2m}),f(a_{31}^{3(i-1)},0,a_{3(i+1)}^{3m}))$

$\hspace{1cm}\in Q-\phi(Q).$\\
Since $Q$ is a  $\phi$-$(2,2)$-absorbing primary hyperideal of $R$, we get

 $\hspace{1cm}g(f(a_{11}^{1(i-1)},0,a_{1(i+1)}^{1m}),f(a_{21}^{2(i-1)},0,a_{2(i+1)}^{2m}))$

$ \hspace{1.5cm}=f(g(a_{11}^{21}),\cdots,g(a_{1(i-1)}^{2(i-1)}),0,g(a_{1(i+1)}^{2(i+1)}),\cdots,g(a_{1m}^{2m}))\subseteq Q$\\ or 

$\hspace{1cm}g(f(a_{21}^{2(i-1)},0,a_{2(i+1)}^{2m}),f(a_{31}^{3(i-1)},0,a_{3(i+1)}^{3m}))$

 $\hspace{1.5cm}=f(g(a_{21}^{31}),\cdots,g(a_{2(i-1)}^{3(i-1)}),0,g(a_{2(i+1)}^{3(i+1)}),\cdots,g(a_{2m}^{3m}))\subseteq {\bf r^{(m,n)}}(Q)$\\
or

$\hspace{1cm}g(f(a_{11}^{1(i-1)},0,a_{1(i+1)}^{1m}),f(a_{31}^{3(i-1)},0,a_{3(i+1)}^{3m}))$

$\hspace{1.5cm}=f(g(a_{11}^{31}),\cdots,g(a_{1(i-1)}^{3(i-1)}),0,g(a_{1(i+1)}^{3(i+1)}),\cdots,g(a_{1m}^{3m}))\subseteq {\bf r^{(m,n)}}(Q)$.\\ 
It implies that 

$\hspace{1cm}f(g(a_{11}^{21}),\cdots,g(a_{1(i-1)}^{2(i-1)}),\phi(Q),g(a_{1(i+1)}^{2(i+1)}),\cdots,g(a_{1m}^{2m}))$

$\hspace{1.5cm}=g(f(a_{11}^{1(i-1)},\phi(Q),a_{1(i+1)}^{1m}),f(a_{21}^{2(i-1)},\phi(Q),a_{2(i+1)}^{2m}))\in Q/\phi(Q)$\\ or  

$\hspace{1cm}f(g(a_{21}^{31}),\cdots,g(a_{2(i-1)}^{3(i-1)}),\phi(Q),g(a_{2(i+1)}^{3(i+1)}),\cdots,g(a_{2m}^{3m}))$

$\hspace{1.5cm}=g(f(a_{21}^{2(i-1)},\phi(Q),a_{2(i+1)}^{2m}),f(a_{31}^{3(i-1)},\phi(Q),a_{3(i+1)}^{3m}))$

$\hspace{1.5cm}\in {\bf r^{(m,n)}}(Q)/\phi(Q)={\bf r^{(m,n)}}(Q/\phi(Q))$\\
or 

$\hspace{1cm}f(g(a_{11}^{31}),\cdots,g(a_{1(i-1)}^{3(i-1)}),\phi(Q),g(a_{1(i+1)}^{3(i+1)}),\cdots,g(a_{1m}^{3m}))$

$\hspace{1.5cm}=g(f(a_{11}^{1(i-1)},\phi(Q),a_{1(i+1)}^{1m}),f(a_{31}^{3(i-1)},\phi(Q),a_{3(i+1)}^{3m}))$

$ \hspace{1.5cm} \in {\bf r^{(m,n)}}(Q)/\phi(Q)={\bf r^{(m,n)}}(Q/\phi(Q))$.

$(2) \Longrightarrow (1) $ Let $g(r_1^3) \in Q-\phi(Q)$ for some $r_1^3 \in R$. Therefore we obtain $f(g(r_1^3),\phi(Q),0^{(m-2)}) \neq \phi(Q)$. It follows that 

 $\phi(Q) \neq g(f(r_1,\phi(Q),0^{(m-2)}),f(r_2,\phi(Q),0^{(m-2)}),f(r_3,\phi(Q),0^{(m-2)})) \in Q/\phi(Q)$. \\
 By the hypothesis, we get 
 
 $g(f(r_1,\phi(Q),0^{(m-2)}),f(r_2,\phi(Q),0^{(m-2)}))=f(g(r_1^2),\phi(Q),0^{(m-2)}) \in Q/\phi(Q)$.\\
 or 
 
 $g(f(r_2,\phi(Q),0^{(m-2)}),f(r_3,\phi(Q),0^{(m-2)}))  =f(g(r_2^3),\phi(Q),0^{(m-2)})$
 
 $\hspace{6.5cm} \in {\bf r^{(m,n)}}(Q)/\phi(Q)$.\\
 or 
 
 $g(f(r_1,\phi(Q),0^{(m-2)}),f(r_3,\phi(Q),0^{(m-2)})) =f(g(r_1^3),\phi(Q),0^{(m-2)})$
 
 $\hspace{6.5cm}  \in {\bf r^{(m,n)}}(Q)/\phi(Q)$.\\
 This shows that $g(r_1^2) \in Q$ or $g(r_2^3) \in {\bf r^{(m,n)}}(Q)$ or $g(r_1^3) \in {\bf r^{(m,n)}}(Q)$. Consequently, $Q$  is   a $\phi$-$(2,2)$-absorbing primary hyperideal of $R$. 
\end{proof}
Suppose that  $I$ is a weakly $(2,2)$-absorbing primary hyperideal of a commutative Krasner $(m,2)$-hyperring $R$. Recall from \cite{weak} that $(x,y,z)$ is said to be $(2,2)$-zero primary of $I$ for $x,y,z \in R$, if $g(x,y,z)=0$, $g(x,y) \notin I$, $g(y,z) \notin  {\bf r^{(m,n)}}(I)$ and $g(x,z) \notin  {\bf r^{(m,n)}}(I)$. 
Now, assume that $Q$ is a $\phi$-$(2,2)$-absorbing primary hyperideal of a commutative Krasner $(m,2)$-hyperring $R$. Then we say $(x,y,z)$ is a $\phi$-$(2,2)$ primary  of $Q$ for some $x,y, z \in R$ if  $g(x,y,z) \in \phi(Q)$, $g(x,y) \notin Q$, $g(y,z) \notin {\bf r^{(m,n)}}(Q)$ and $g(x,z) \notin {\bf r^{(m,n)}}(Q)$. It is easy
to see that  a proper hyperideal $Q$ of $R$ is $\phi$-$(2,2)$-absorbing primary that is not $(2,2)$-absorbing primary if and only if $Q$ has a $\phi$-$(2,2)$ $(x,y,z)$ for some $x,y,z \in R$.
\begin{theorem} \label{weak2}
Let $R$ be a commutative Krasner $(m,2)$-hyperring  and let $\phi: \mathcal{HI}(R) \longrightarrow \mathcal{HI}(R) \cup \{\varnothing\}$ be a function. Let $Q$ be a $\phi$-$(2,2)$-absorbing primary hyperideal of $R$ and   $x,y,z \in R$. Then the followings are equivalent:
\begin{itemize} 
\item[\rm(1)]~ $(x,y,z)$ is a $\phi$-$(2,2)$ primary of $Q$.
\item[\rm(2)]~$(f(x,\phi(Q),0^{(m-2)}),f(y,\phi(Q),0^{(m-2)}),f(z,\phi(Q),0^{(m-2)})$ is a $(2,2)$-zero primary of $Q/\phi(Q)$.
\end{itemize}
\end{theorem}
\begin{proof}
$(1) \Longrightarrow (2)$ Let $(x,y,z)$ be a $\phi$-$(2,2)$ primary of $Q$. This means that $g(x,y,z) \in \phi(Q)$, $g(x,y) \notin Q$, $g(y,z) \notin {\bf r^{(m,n)}}(Q)$ and $g(x,z) \notin {\bf r^{(m,n)}}(Q)$. This implies that $f(g(x,y),Q,0^{(m-2)}) \notin Q/\phi(Q)$, $f(g(y,z),\phi(Q),0^{(m-2)}) \notin {\bf r^{(m,n)}}(Q)/\phi(Q)$ and $f(g(x,z),\phi(Q),0^{(m-2)}) \notin {\bf r^{(m,n)}}(Q)/\phi(Q)$. By Theorem \ref{weak}, we conclude that $(f(x,\phi(Q),0^{(m-2)}),f(y,\phi(Q),0^{(m-2)}),f(z,\phi(Q),0^{(m-2)})$ is a $(2,2)$-zero primary of $Q/\phi(Q)$.

$(2) \Longrightarrow (1)$ Assume that $(f(x,\phi(Q),0^{(m-2)}),f(y,\phi(Q),0^{(m-2)}),f(z,\phi(Q),0^{(m-2)})$ is a $(2,2)$-zero primary of $Q/\phi(Q)$. Thus $g(x,y,z) \in \phi(Q)$ but $f(g(x,y),Q,0^{(m-2)}) \notin Q/\phi(Q)$, $f(g(y,z),\phi(Q),0^{(m-2)}) \notin {\bf r^{(m,n)}}(Q)/\phi(Q)$ and $f(g(x,z),\phi(Q),0^{(m-2)}) \notin {\bf r^{(m,n)}}(Q)/\phi(Q)$. Hence $g(x,y,z) \in \phi(Q)$, $g(x,y) \notin Q$, $g(y,z) \notin {\bf r^{(m,n)}}(Q)$ and $g(x,z) \notin {\bf r^{(m,n)}}(Q)$. It implies that $(x,y,z)$ is a $\phi$-$(2,2)$ primary of $Q$.
\end{proof}

\begin{theorem} \label{weak3}
Let $R$ be a commutative Krasner $(m,2)$-hyperring  and let $\phi: \mathcal{HI}(R) \longrightarrow \mathcal{HI}(R) \cup \{\varnothing\}$ be a function. Let $Q$ be a $\phi$-$(2,2)$-absorbing primary hyperideal of $R$. If $(x,y,z)$ is a $\phi$-$(2,2)$ primary of $Q$ for some $x,y,z \in R$, then
\begin{itemize} 
\item[\rm(1)]~ $g(x,y,Q), g(y,z,Q), g(x,z,Q) \subseteq \phi(Q)$.
\item[\rm(2)]~$g(x,Q^{(2)}), g(y,Q^{(2)}), g(z,Q^{(2)}) \subseteq \phi(Q)$.
\item[\rm(3)]~$g(Q^{(3)}) \subseteq \phi(Q)$.
\end{itemize} 
\end{theorem}
\begin{proof}
(1) Let $(x,y,z)$ be a $\phi$-$(2,2)$ primary of a $\phi$-$(2,2)$-absorbing primary hyperideal $Q$. By Theorem \ref{weak2},  $(f(x,\phi(Q),0^{(m-2)}),f(y,\phi(Q),0^{(m-2)}),f(z,\phi(Q),0^{(m-2)})$ is a $(2,2)$-zero primary of $Q/\phi(Q)$ since $(x,y,z)$ is a $\phi$-$(2,2)$ primary of $Q$.  Thus 

$\hspace{1cm}f(g(x,y,Q),\phi(Q),0^{(m-2)})=f(g(y,z,Q),\phi(Q),0^{(m-2)})$

$\hspace{5.1cm}=f(g(x,z,Q),\phi(Q),0^{(m-2)})$

$\hspace{5.1cm}=\phi(Q)$\\ by Theorem 4.9 in \cite{weak}, which implies $g(x,y,Q)$, $g(y,z,Q)$ and $g(x,z,Q)$ are subsets of $\phi(Q)$.

(2) Theorem \ref{weak2} shows that $(f(x,\phi(Q),0^{(m-2)}),f(y,\phi(Q),0^{(m-2)}),f(z,\phi(Q),0^{(m-2)})$ is a $(2,2)$-zero primary of $Q/\phi(Q)$. Moreover, Theorem \ref{weak} shows that $Q/\phi(Q)$ is a weakly $(2,2)$-absorbing primary of $R/\phi(Q)$. Then $f(g(x,Q^{(2)}),\phi(Q),0^{(m-2)})=f(g(y,Q^{(2)}),\phi(Q),0^{(m-2)})=f(g(z,Q^{(2)}),\phi(Q),0^{(m-2)})=\phi(Q)$, by Theorem 4.9 of \cite{weak}. Consequently, $g(x,Q^{(2)}), g(y,Q^{(2)}), g(z,Q^{(2)})$ are subsets of $\phi(Q)$.

(3) Again,  $(f(x,\phi(Q),0^{(m-2)}),f(y,\phi(Q),0^{(m-2)}),f(z,\phi(Q),0^{(m-2)})$ is a $(2,2)$-zero primary of $Q/\phi(Q)$ and $Q/\phi(Q)$ is a weakly $(2,2)$-absorbing primary of $R/\phi(Q)$ by Theorem \ref{weak2} and Theorem \ref{weak}, respectively, then $f(g(Q^{(3)}),\phi(Q),0^{(m-2)})=\phi(Q)$ by Theorem 4.10 in \cite{weak}. Thus $g(Q^{(3)})$ is a subset of $\phi(Q)$.
\end{proof}
\begin{theorem} \label{kweak}
Suppose that $Q$ is a proper hyperideal of a commutative Krasner $(m,n)$-hyperring $R$ and $\phi: \mathcal{HI}(R) \longrightarrow \mathcal{HI}(R) \cup \{\varnothing\}$ is a function. Then the followings are equivalent:

\begin{itemize} 
\item[\rm(1)]~ $Q$  is   a $\phi$-$(k,n)$-absorbing primary hyperideal of $R$.  
\item[\rm(2)]~ $Q/\phi(Q)$ is a weakly $(k,n)$-absorbing primary hyperideal of $R/\phi(Q)$.



\end{itemize}
\end{theorem}
\begin{proof}
It can be easily proved  in a similar manner to the proof
of Theorem \ref{weak}.
\end{proof}
Suppose that $Q$ is a $\phi$-$(k,n)$-absorbing primary hyperideal of  $R$. Then we say $(r_1^{k(n-1)+1})$ is a $\phi$-$(k,n)$ primary  of $Q$ for some $r_1^{k(n-1)+1} \in R$ if  $g(r_1^{k(n-1)+1}) \in \phi(Q)$, $g(r_1^{(k-1)n-k+2}) \notin Q$ and a $g$-product of $(k-1)n-k+2$ of $r_i^,$s, except $g(r_1^{(k-1)n-k+2})$, is not in ${\bf r^{(m,n)}}(Q)$.
\begin{theorem} \label{kweak2}
Let $R$ be a commutative Krasner $(m,2)$-hyperring  and let $\phi: \mathcal{HI}(R) \longrightarrow \mathcal{HI}(R) \cup \{\varnothing\}$ be a function. Let $Q$ be a $\phi$-$(k,n)$-absorbing primary hyperideal of $R$ and  $r_1^{k(n-1)+1} \in R$. Then the followings are equivalent:
\begin{itemize} 
\item[\rm(1)]~ $(r_1^{k(n-1)+1})$ is a $\phi$-$(k,n)$ primary of $Q$.
\item[\rm(2)]~$(f(r_1,\phi(Q),0^{(m-2)}),\cdots,f(r_{k(n-1)+1},\phi(Q),0^{(m-2)})$ is a $(k,n)$-zero primary of $Q/\phi(Q)$.
\end{itemize}
\end{theorem}
\begin{proof}
It is seen to be true in a similar manner to Theorem \ref{weak2}.
\end{proof}
\begin{theorem} \label{weak4}
Let $R$ be a commutative Krasner $(m,n)$-hyperring  and let $\phi: \mathcal{HI}(R) \longrightarrow \mathcal{HI}(R) \cup \{\varnothing\}$ be a function. Let $Q$ be a $\phi$-$(k,n)$-absorbing primary hyperideal of $R$. If $(r_1^{k(n-1)+1})$ is a $\phi$-$(k,n)$ primary of $Q$ for some $r_1^{k(n-1)+1} \in R$, then $g(r_1,\cdots,\widehat{r_{i_1}},\cdots,\widehat{r_{i_2}},\cdots,\widehat{r_{i_s}},\cdots,r_{k(n-1)+1},Q^{(s)}) \subseteq \phi(Q)$ for every $i_1,\cdots,i_s \in \{1,\cdots,k(n-1)+1\}$ and $1 \leq s \leq (k-1)n-k+2$.
\end{theorem}
\begin{proof}
 $(f(r_1,\phi(Q),0^{(m-2)}),\cdots,f(r_{k(n-1)+1},\phi(Q),0^{(m-2)})$ is a $(k,n)$-zero primary of $Q/\phi(Q)$ by Theorem \ref{kweak2} and   $Q/\phi(Q)$ is a weakly $(k,n)$-absorbing primary of $R/\phi(Q)$ by Theorem \ref{kweak}. Then  we conclude that 

$f(g(f(r_1,\phi(Q),0^{(m-2)}),\cdots,f(\widehat{r_{i_1}},\phi(Q),0^{(m-2)}),\cdots,f(\widehat{r_{i_2}},\phi(Q),0^{(m-2)}),\cdots,$

$\hspace{0.1cm}f(\widehat{r_{i_s}},\phi(Q),0^{(m-2)}),\cdots,
 f(r_{k(n-1)+1},\phi(Q),0^{(m-2)}),Q^{(s)}),\phi(Q),0^{(m-2)})= \phi(Q)$\\ for every $i_1,\cdots,i_s \in \{1,\cdots,k(n-1)+1\}$ and $1 \leq s \leq (k-1)n-k+2$, by Theorem 4.9 of \cite{weak}. Thus, $g(r_1,\cdots,\widehat{r_{i_1}},\cdots,\widehat{r_{i_2}},\cdots,\widehat{r_{i_s}},\cdots,r_{k(n-1)+1},Q^{(s)}) \subseteq \phi(Q)$.
\end{proof}
\begin{theorem} \label{weak5}
Let $R$ be a commutative Krasner $(m,n)$-hyperring  and let $\phi: \mathcal{HI}(R) \longrightarrow \mathcal{HI}(R) \cup \{\varnothing\}$ be a function. Let $Q$ be a $\phi$-$(k,n)$-absorbing primary hyperideal of $R$ but is not a $(k,n)$-absorbing primary. Then $g(Q^{k(n-1)+1}) \subseteq \phi(Q)$.
\end{theorem}
\begin{proof}
This can be proved, by using Theorem \ref{weak4}, in a very similar manner to the way in which \ref{weak3} was proved.
\end{proof}
Now, let give some related corollaries.
\begin{corollary}
Let $\phi: \mathcal{HI}(R) \longrightarrow \mathcal{HI}(R) \cup \{\varnothing\}$ be a function. If $Q$ is a $\phi$-$(k,n)$-absorbing primary hyperideal of $R$ such that $g(Q^{k(n-1)+1}) \nsubseteq \phi(Q)$, then $Q$ is a $(k,n)$-absorbing primary hyperideal of $R$.
\end{corollary}
\begin{corollary} \label{weak6}
Let $\phi: \mathcal{HI}(R) \longrightarrow \mathcal{HI}(R) \cup \{\varnothing\}$ be a function and let $Q$ be a $\phi$-$(k,n)$-absorbing primary hyperideal of $R$  that is not a $(k,n)$-absorbing primary hyperideal of $R$. Then ${\bf r^{(m,n)}}(Q)={\bf r^{(m,n)}}(\phi(Q))$.
\end{corollary}
\begin{proof}
By Theorem \ref{weak5}, we have $g(Q^{k(n-1)+1}) \subseteq \phi(Q)$ as  $Q$ is not a $(k,n)$-absorbing primary. This means ${\bf r^{(m,n)}}(Q) \subseteq {\bf r^{(m,n)}}(\phi(Q))$. On the other hand, from $\phi(Q) \subseteq Q$, it follows that ${\bf r^{(m,n)}}(\phi(Q)) \subseteq {\bf r^{(m,n)}}(Q)$. Hence ${\bf r^{(m,n)}}(Q)={\bf r^{(m,n)}}(\phi(Q))$.
\end{proof}
\begin{corollary}
Let $\phi: \mathcal{HI}(R) \longrightarrow \mathcal{HI}(R) \cup \{\varnothing\}$ be a function and let $Q$ be a proper hyperideal of $R$ such that ${\bf r^{(m,n)}}(\phi(Q))$ is a $(k,n)$-absorbing hyperideal of $R$. Then $Q$ is a $\phi$-$(k+1,n)$-absorbing primary hyperideal of $R$ if and only if $Q$ is a $(k+1,n)$-absorbing primary hyperideal of $R$.
\end{corollary}
\begin{proof}
$(\Longrightarrow)$ Let $Q$ be a $\phi$-$(k+1,n)$-absorbing primary hyperideal of $R$. If $Q$ is not a $(k+1,n)$-absorbing primary hyperideal of $R$. Hence ${\bf r^{(m,n)}}(Q)={\bf r^{(m,n)}}(\phi(Q))$ by Corollary \ref{weak6}. Then ${\bf r^{(m,n)}}(Q)$ is a $(k,n)$-absorbing hyperideal of $R$ which implies that $Q$ is is a $(k+1,n)$-absorbing primary hyperideal of $R$ by Theorem 4.9 in \cite{rev2}. 

$(\Longleftarrow)$ It is clear.
\end{proof}
\begin{theorem} \label{homo}
Let $h:R_1 \longrightarrow R_2$ be a $\phi_1$-$\phi_2$-homomorphism, where $\phi_1$ and $\phi_2$ are two reduction functions of $\mathcal{HI}(R_1)$ and $\mathcal{HI}(R_2)$, respectively. Then:
\begin{itemize} 
\item[\rm(1)]~ If $Q_2$ is a $\phi_2$-$(k,n)$-absorbing primary hyperideal of $R_2$, then $h^{-1}(Q_2)$ is a $\phi_1$-$(k,n)$-absorbing primary hyperideal of $R_1$. 
\item[\rm(2)]~If $h$ is surjective and $Q_1$ is a $\phi_1$-$(k,n)$-absorbing primary hyperideal of $R_1$ with $Ker (h) \subseteq Q_1$, then $h(Q_1)$ is a $\phi_2$-$(k,n)$-absorbing primary hyperideal of $R_2$.
\end{itemize} 
\end{theorem}
\begin{proof}
 $(1)$ Let $Q_2$ be a $\phi_2$-$(k,n)$-absorbing primary hyperideal of $R_2$. Assume that $r_1^{kn-k+1} \in R_1$ such that $g(r_1^{kn-k+1}) \in h^{-1}(Q_2)-\phi_1(h^{-1}(Q_2))$. Then we get $h(g(r_1^{kn-k+1}))=g(h(r_1),\cdots,h(r_{kn-k+1})) \in Q_2-\phi_2(Q_2)$. Since $Q_2$ is a $\phi_2$-$(k,n)$-absorbing primary hyperideal of $R_2$, we obtain either $g(h(r_1),\cdots,h(r_{(k-1)n-k+2}))=h(g(r_1^{(k-1)n-k+2})) \in Q_2$ which means $g(r_1^{(k-1)n-k+2}) \in h^{-1}(Q_2)$, or 
 
 $g(h(r_1),\cdots,\widehat{h(r_i)},\cdots,h(r_{kn-k+1)})=h(g(r_1,\cdots,\widehat{r_i},\cdots,r_{kn-k+1)})) \in {\bf r^{(m,n)}}(Q_2)$
 which means $g(r_1,\cdots,\widehat{r_i},\cdots,r_{kn-k+1)}) \in h^{-1}({\bf r^{(m,n)}}(Q_2))={\bf r^{(m,n)}}(h^{-1}(Q_2))$ for some $1 \leq i \leq n$. Hence $h^{-1}(Q_2)$ is a $\phi_1$-$(k,n)$-absorbing primary hyperideal of $R_1$.

 $(2)$ Assume that $h$ is surjective and $Q_1$ is a $\phi_1$-$(k,n)$-absorbing primary hyperideal of $R_1$ with $Ker (h) \subseteq Q_1$. Let $s_1^{kn-k+1} \in R_2$ such that $g(s_1^{kn-k+1}) \in h(Q_1)-\phi_2(h(Q_1))$. Therefore there exist $r_1^{kn-k+1} \in R_1$ with $h(r_1)=s_1, \cdots, h(r_{kn-k+1})=s_{kn-k+1}$. Hence we get $h(g(r_1^{kn-k+1})=g(h(r_1),\cdots,h(r_{kn-k+1}))=g(s_1^{kn-k+1}) \in h(Q_1)$. Since $h$ is a $\phi_1$-$\phi_2$-epimorphism and $Ker (h) \subseteq Q_1$, we have $g(r_1^{kn-k+1}) \in Q_1-\phi_1(Q_1)$. Since $Q_1$ is a $\phi_1$-$(k,n)$-absorbing primary hyperideal of $R_1$, we conclude that $g(r_1^{(k-1)n-k+2)}) \in Q_1$ which implies 
 
$ \hspace{2cm}h(g(r_1^{(k-1)n-k+2)})=g(h(r_1),\cdots,h(r_{(k-1)n-k+2)}) $

$\hspace{4.8cm}=g(s_1^{(k-1)n-k+2)}) \in h(Q_1)$,\\ 
 or $g(r_1,\cdots,\widehat{r_i},\cdots,r_{kn-k+1}) \in {\bf r^{(m,n)}}(Q_1)$  implies $h(g(r_1,\cdots,\widehat{r_i},\cdots,r_{kn-k+1})=g(h(r_1),\cdots,\widehat{h(r_i)},\cdots,h(r_{kn-k+1}))=g(s_1,\cdots,\widehat{s_i},\cdots,s_{kn-k+1}) \in h({\bf r^{(m,n)}}(Q_1))\subseteq {\bf r^{(m,n)}}(h(Q_1))$ for some $1 \leq i \leq (k-1)n-k+2$. Consequently,  $h(Q_1)$ is a $\phi_2$-$(k,n)$-absorbing primary hyperideal of $R_2$.
\end{proof}
 As an instant consequence of the previous theorem, we get the following explicit result.

\begin{theorem}
Let $Q$ and $P$ be  two  hyperideals of $R$ and $\phi$ be a reduction function of $\mathcal{HI}(R)$ such that $P \subseteq \phi(Q) \subseteq Q$. If $Q$ is a $\phi$-$(k,n)$-absorbing primary hyperideal of $R$, then $Q/P$ is a $\phi_q$-$(k,n)$-absorbing primary hyperideal of $R/P$.
\end{theorem}
\begin{theorem}
 Let $(R_i,f_i,g_i)$ be a commutative
Krasner $(m,n)$-hyperring  for each $1 \leq i \leq kn-k+1$ and $\phi_i: \mathcal{HI}(R_i) \longrightarrow \mathcal{HI}(R_i) \cup \{\varnothing\}$ be a function. Let $Q_i$ be a  hyperideal of $R_i$ for each $1 \leq i \leq kn-k+1$ and $\phi=\phi_1 \times \cdots \times \phi_{kn-k+1}$. If $Q=Q_1 \times \cdots \times Q_{kn-k+1}$ is a $\phi$-$(k+1,n)$-absorbing primary hyperideal of $R=R_1 \times \cdots \times R_{kn-k+1}$, then $Q_i$ is a $\phi_i$-$(k,n)$-absorbing primary  hyperideal of $R_i$ and $Q_i \neq R_i$ for all $1 \leq i \leq kn-k+1$. 
 \end{theorem}
\begin{proof}
By using an argument similar to that in the proof of Theorem \ref{car},
one can easily complete the proof.
\end{proof}
\section{Conclusion}
In this paper, motivated by the research works on $\phi$-2-absorbing (primary) ideals of  commutative rings, we propsed and  investigated the notions of $\phi$-$(k,n)$-absorbing and $\phi$-$(k,n)$-absorbing primary hyperideals in a Krasner $(m,n)$-hyperring.
 Some of their essential characteristics were analysed.
Moreover,  the stabilty of the notions were examined in some
hyperring-theoretic constructions. 
As a new research subject, we suggest the concept of $\phi$-$(k,n)$-absorbing $\delta$-primary hyperideals, where $\delta$ is an expansion function of $\mathcal{HI}(R)$.

\end{document}